\documentclass[reqno]{amsart}
\usepackage{amssymb}
\usepackage{hyperref}

\title[ ]
{$C_0$ Semigroup And Local Spectral Theory }

\author[ A. Tajmouati, M. Karmouni,  H. Boua,  Z. Al-homidi ]
{  A. Tajmouati, M. Karmouni,  H. Boua,  Z. Al-homidi }

\address{A. Tajmouati, M. Karmouni,  H. Boua,  Z. Al-homidi \, \newline
 Sidi Mohamed Ben Abdellah
 Univeristy
 Faculty of Sciences Dhar Al Mahraz Fez, Morocco.}
\email{abdelaziztajmouati@yahoo.fr}
\email{mohammed.karmouni@usmba.ac.ma}
\email{hamid12boua@yahoo.com}
\email{zakariya1978@yahoo.com}

\subjclass[2000]{47D03, 47A10, 47A11}
\keywords{$C_0$ Semigroup, local spectrum, SVEP, stability.}

\newtheorem{theorem}{Theorem}[section]

\newtheorem{lemma}{Lemma}[section]

\newtheorem{corollary}{Corollary}[section]

\begin{document}
\maketitle
\begin{abstract}
In this paper, we studied some local spectral properties for a $C_0$ semigroup and its generator.
Some stabilities  results are also established.
\end{abstract}

\section{Introduction}
 The semigroups can be used to solve a large class of problems commonly known as the Cauchy problem:

 $$u'(t)=Au(t), t\geq 0, u(0)=u_0  \>\>\>\>\>\>\>(1) $$

 on  a Banach space $X$. Here $A$ is a given linear operator with domain $D(A)$ and the initial value $u_0$. The solution of (1) will be given by $u(t)=T(t)u_0$ for an operator semigroup $(T(t))_{t\geq 0}$ on $X$. In this paper,  We will focus on a special class of linear
semigroups called $C_0$ semigroups which are semigroups of strongly continuous bounded linear operators. Precisely:\\
 A one-parameter family $(T(t))_{t\geq0}$ of operators on  a Banach space $X$ is called a $C_{0}$-semigroup of operators or a strongly continuous semigroup of operators if:
  \begin{enumerate}
    \item $T(0) = I$.
    \item $T(t + s) = T(t)T(s)$, $\forall t, s\geq 0$.
    \item $\displaystyle{\lim_{t\rightarrow 0}} T(t)x = x$, $\forall x\in X$.
  \end{enumerate}
 $(T(t))_{t\geq0}$ has a unique infinitesimal generator $A$ defined in domain $D(A)$ by:\\
$$D(A)=\{x\in X : \displaystyle{\lim_{t\rightarrow 0}} \frac{T(t)x-x}{t} \mbox{ exists} \}$$,
 $$Ax = \displaystyle{\lim_{t\rightarrow 0}} \frac{T(t)x-x}{t}, \forall x\in D(A)$$\\
 Recall that for all $t\geq 0$,   $T(t)$ is a bounded  linear operator on $X$ and $A$ is a closed operator. Details for all this may be found in  \cite{Pazy,engel}.
There are enough studies done on semi-groups, including spectral studies \cite{PA, Nagel, Mulk, Mulko, engel, Pazy, TB}. In this article, we investigate the transmission of some local spectral properties from a $C_0$ semi-group to its infinitesimal generator.

\section{Preliminaries}
Throughout, $X$ denotes a complex Banach space, let $A$ be a closed linear operator on $X$ with domain  $D(A)$, we denote by $A^*$,  $R(A)$, $N(A)$, $ R^{\infty}(A)=\bigcap_{n\geq0}R(A^n)$, $\sigma_{K}(A)$,  $\sigma_{su}(A)$,  $\sigma(A)$,
 respectively the adjoint, the range, the null space, the hyper-range, the semi-regular spectrum,  the surjectivity spectrum and the spectrum of $A$.\\

Recall that for  a closed linear operator $A$ and $x\in X$ the local resolvent of $A$ at $x$,  $\rho_{A}(x)$ defined as the union  of all  open subset $U$ of $\mathbb{C}$ for which there is an analytic function $f: U\rightarrow D(A)$ such that the equation $(A-\mu I)f(\mu)=x$ holds for
all $ \mu \in U$. The local spectrum $\sigma_A(x)$ of $A$
 at $x$ is defined as $\sigma_A(x)=\mathbb{C}\setminus \rho_A(x)$. Evidently  $\sigma_A(x)\subseteq\sigma_{su}(A)\subseteq\sigma(A)$, $\rho_A(x)$ is open and $\sigma_A(x)$ is closed.\\
If $f(z)=\displaystyle{\sum_{i=0}^{\infty}}x_i(z-\mu)^i$  ( in  a neighborhood of $\mu$),   be the Taylor expansion of $f$, it is easy to see that $\mu\in\rho_{A}(x)$ if and only if there exists a sequence $(x_i)_{i\geq0}\subseteq D(A)$, $(A-\mu)x_0=x$,  $(A-\mu)x_{i+1}=x_i$, and $sup_i||x_i||^{\frac{1}{i}}<\infty$, see \cite{E, Mul}.\\
The local spectral subspace of $A$  associated with a subset $\Omega$ of $\mathbb{C}$ is the set :
 $$X_A(\Omega)=\{x\in X: \sigma_A(x)\subseteq \Omega\}$$
  Evidently $X_A(\Omega)$ is a hyperinvariant subspace of $A$ not always closed.\\

Next, let $A$ a closed  linear operator, $A$ is said to have the single
valued extension property at $\lambda_{0}\in\mathbb{C}$ (SVEP) if
for every  open neighborhood   $U\subseteq \mathbb{C}$ of
$\lambda_{0}$, the only  analytic function  $f: U\longrightarrow D(A)$ which satisfies
 the equation $(A-zI)f(z)=0$ for all $z\in U$ is the function $f\equiv 0$. $A$ is said to have the SVEP if $A$ has the SVEP for
 every $\lambda\in\mathbb{C}$. Denote by $$S(A)=\{\lambda\in \mathbb{C}: A\mbox{ has not  the SVEP at } \lambda\}.$$  $S(A)$ is an open of $\mathbb{C}$
  and  $X_A(\emptyset)=\{0\}$ implies $S(A)=\emptyset$ \cite{PA}. If $A$ is bounded, then
 $X_A(\emptyset)$ is closed if and only if $X_A(\emptyset)=\{0\}$ if and only if $S(A)=\emptyset$ \cite{lau}.\\
Note that $\mu\in S(A)$ if and only if there exists a sequence $(x_i)_{i\geq0}\subseteq D(A)$ not all of them equal to zero such that
 $(A-\mu)x_{i+1}=x_i$, with $x_0=0$  and $sup_i||x_i||^{\frac{1}{i}}<\infty$, see \cite{E}.\\
For  a closed linear operator $A$ the algebraic core $C(A)$ for $A$ is the greatest subspace $M$ of $X$ for which $A(M\cap D(A))=M$. Equivalently:
$$C(A) = \{x \in D(A) : \,\,\,\, \exists (x_n)_{n\geq0} \subset  D(A) ,\,\,\mbox{ such that } x_{0}=x,   A x_n = x_{n-1} \mbox{ for all } n\geq1\}$$

Moreover the analytic core for $A$ is a linear subspace of $X$
defined by:
\begin{center}
 $K(A) = \{x \in D(A) :  \exists (x_n)_{n\geq0} \subset  D(A) \mbox{ and } \delta > 0 \mbox{ such that } x_{0}=x, A x_n = x_{n-1} \,\,\,\,\forall n \geq1 \mbox{ and } \|x_{n}\| \leq \delta^{n}\|x\|\}$
 \end{center}
The analytica core admits a local spectral characterization for unbounded operator as follow \cite[ Theorem 4.3]{PA}:
$$K(A)=\{x\in D(A): 0\in \rho_{A}(x)\}=X_{A}(\mathbb{C}\setminus\{0\}\}$$
 Note that in general neither  $K(T)$ nor $C(T)$ are closed and we have $$X_{A}(\emptyset)\subset K(A)\subseteq C(A)\subset R^{\infty}(A) \subset R(A).$$

 Let $(T(t))_{t\geq0}$, a $C_0$ semigroup with infinitesimal generator $A$,  we introduce the following operator acting on $X$ and depending on the parameters $\lambda \in \mathbb{C}$ and $t\geq 0$:
\begin{center}
$B_\lambda(t)x=\int_0^t e^{\lambda (t-s)}T(s)xds,  x\in X  $
\end{center}
It is well known  that $B_\lambda(t)$ is a bounded linear operator on $X$ \cite{Pazy, engel} and we have:
\begin{center}
$(e^{\lambda t}-T(t))^nx =(\lambda -A)^nB^n_\lambda(t)x,\;\;\;   \forall x\in X,  n\in\mathbb{N}$ \\
$(e^{\lambda t}-T(t))^nx =B^n_\lambda(t)(\lambda -A)^n x,\;\;\;   \forall x\in D(A^n),   n\in\mathbb{N}$;\\
      $R^{\infty}(e^{\lambda t}-T(t))\subseteq R^{\infty}(\lambda -A)$;\\
      $N((\lambda -A)^n)\subseteq N(e^{\lambda t}-T(t))^n.$
\end{center}
In \cite{Pazy,engel}, they showed that:
$$e^{t\nu (A)}\subseteq  \nu(T(t))$$
where $\nu \in \{\sigma_p,\sigma_{ap} ,\sigma_r\}$, point spectrum, approximative spectrum and residual spectrum.\\
In \cite{Mulk},  the authors showed this spectral inclusion for the semi-regular spectrum.\\
In this work, as a continuous of the previous work, we will give a spectral inclusion for  local spectrum.  Also,  we investigate some  local spectral properties for  $C_0$ semigroup  and its generator. Some stabilities results are established.


\section{ Local Spectral Theory  }
Now, we start the present section by the following lemma which we need in the sequel.
\begin{lemma}\label{1}\cite{TB}
Let $(T(t))_{t\geq 0}$ a $C_0$-semigroup on $X$ with infinitesimal generator $A$. For  $\lambda \in \mathbb{C}$ and $t\geq 0$, let $F_{\lambda}(t)x=\int_0^t e^{-\lambda s}B_\lambda(s)x ds$, then:
\begin{enumerate}
\item There exist a  $M\geq 1$ and $\omega > Re(\lambda)$ such that $F_{\lambda}(t) \leq \frac{M}{(\omega-Re(\lambda))^2} e^{(\omega-Re(\lambda)) t} .$
\item  $\forall x\in X$, $F_{\lambda}(t)x\in D(A)$ and  $(\lambda -A)F_{\lambda}(t) + G_{\lambda}(t) B_\lambda(t)=tI$ with $G_{\lambda}(t)=e^{-\lambda t}I.$
\item The operators $F_{\lambda}(t)$, $G_{\lambda}(t)$ and  $B_\lambda(t)$ are pairwise commute and for all  $x\in D(A)$:
\begin{eqnarray*}
  (\lambda -A)F_{\lambda}(t)x&=& F_{\lambda}(t)(\lambda -A)x \\
  (\lambda -A)G_{\lambda}(t)x&=& G_{\lambda}(t)(\lambda -A)x \\
    (\lambda -A)B_\lambda(t)x&=&B_\lambda(t)(\lambda -A)x
\end{eqnarray*}
\end{enumerate}
\end{lemma}

\begin{theorem}\label{1}
For the generator $ A $ of a strongly continuous semigroup
$(T (t))_{t\geq0}$, we have the following  inclusion :
$$S(T(t))\subseteq e^{tS(A)} .$$
\end{theorem}
\begin{proof}
Suppose that $e^{\lambda t}-T(t)$ has not  SVEP at $0$, then there exist $x_i\in X$ such that $x_0=0$, $(e^{\lambda t}-T(t))x_i=x_{i-1}$
and $\sup_{i}\|x_i\|^{\frac{1}{i}}<\infty.$
Let $y_i=B_{\lambda}^{i}(t)x_i$, then $(y_i)_{i\geq 0}\subseteq  D(A)$ and $y_0=x_0=0$, and we have :

\begin{eqnarray*}
  (\lambda-A)y_i &=& (\lambda-A)B_{\lambda}(t) B_{\lambda}^{i-1}(t)x_i \\
   &=& (e^{\lambda t}-T(t))B_{\lambda}^{i-1}(t)x_i \\
   &=& B_{\lambda}^{i-1}(t)(e^{\lambda t}-T(t))x_i \\
   &=& B_{\lambda}^{i-1}(t)x_{i-1}\\
   &=& y_{i-1}
\end{eqnarray*}

Therefore,  $(\lambda-A)y_i=y_{i-1}.$
On the other hand $\|y_i\|=\| B_{\lambda}^{i}(t)x_i\|\leq\|B_{\lambda}^{i}(t)\| \|x_i\|\leq M^i\|x_i\|$,
then $\sup_{i}\|y_i\|^{\frac{1}{i}}\leq\sup_{i} M\|x_i\|^{\frac{1}{i}}<\infty$, $M>0$.
So that $\lambda -A$ has not SVEP at $0$, then $S(T(t))\subseteq e^{tS(A)}$.
\end{proof}
In the following, we give a sufficient condition to show that the local spectral subspace $X_{T(t)}(\emptyset)$, $t>0$,  is closed.
\begin{corollary}
Let $(T(t))_{t\geq0}$ a $C_0-$semigroup, with generator $A$. Then:
\begin{center}
$X_{A}(\emptyset)=\{0\}$ implies that $X_{T(t)}(\emptyset)$ is closed for all $t\geq 0$.
\end{center}
\end{corollary}
\begin{proof}
Since $X_{A}(\emptyset)=\{0\}$ implies that $S(A)=\emptyset$, by theorem \ref{1} we have $S(T(t))=\emptyset$ which equivalent to the fact that $X_{T(t)}(\emptyset)=\{0\}$ equivalently to  $X_{T(t)}(\emptyset)$ is closed.
\end{proof}
\begin{corollary}
Let $(T(t))_{t\geq0}$ a $C_0-$semigroup, with generator $A$. If $A$ has the SVEP then $T(t)$ has the SVEP for all $t\geq 0$.
\end{corollary}
In the following Theorem, we give a spectral inclusion for the local spectrum.
\begin{theorem}\label{T1}
Let $(T(t))_{t\geq 0}$ a $C_0$-semigroup on $X$ with infinitesimal generator $A$.  The following spectral inclusion hold :
$$e^{t\sigma_{A}(x)}\subseteq \sigma_{T(t)}(x)$$
\end{theorem}
\begin{proof}
Let  $e^{\lambda t}\notin \sigma_{T(t)}(x)$, then there exists $(x_i)_{i\geq 0}\subseteq X, $
such that $$(e^{\lambda t}-T(t))x_0=x,\>\> (e^{\lambda t}-T(t))x_i=x_{i-1} \mbox{  and  } \sup\|x_i\|^{\frac{1}{i}}<\infty.$$
Let $y_i=B_{\lambda}^{i+1}(t)x_i$, then $(y_i)_{i\geq 0}\subseteq  D(A)$ and   $y_0=B_{\lambda}x_0$.
 We have : \begin{eqnarray*}
  (\lambda-A)y_i &=& (\lambda-A)B_{\lambda}(t) B_{\lambda}^{i}(t)x_i \\
   &=& (e^{\lambda t}-T(t))B_{\lambda}^{i}(t)x_i \\
   &=& B_{\lambda}^{i}(t)(e^{\lambda t}-T(t))x_i \\
   &=& B_{\lambda}^{i}(t)x_{i-1}\\
   &=& y_{i-1}
\end{eqnarray*}
 and   $$\sup\|y_i\|^{\frac{1}{i}}<\infty$$
So that $\lambda \notin \sigma_A(x)$\\
\end{proof}
\begin{corollary}\label{c1}
Let $(T(t))_{t\geq0}$ a $C_0-$semigroup, with generator $A$. Then:
\begin{enumerate}
  \item $K(e^{\lambda t}- T(t))\subseteq K(\lambda-A)$;
  \item $C(e^{\lambda t}- T(t))\subseteq C(\lambda-A)$.
\end{enumerate}
\end{corollary}

\begin{proof}
Since $K(e^{\lambda t}- T(t))=\{x\in X; 0\in\rho_{(e^{\lambda t}- T(t))}(x)\}$, then if $x\in K(e^{\lambda t}- T(t))$ implies that $e^{\lambda t}\in \rho_{T(t)}(x)$, by theorem \ref{T1} we have $\lambda\in \rho_{A}(x)$, therefore $x\in K(\lambda-A)$.
\end{proof}
Denote by $\sigma_{ac}(T)=\{\lambda\in\mathbb{C}: K(\lambda-T)=\{0\}\}$ the analytic core spectrum of $T$ and by
$\sigma_{alc}(T)=\{\lambda\in\mathbb{C}:  C(\lambda-T)=\{0\}\}$ the  algebraic core spectrum of $T$ \cite{bak,bakk}.
As a straightforward consequence of the corollary \ref{c1}, we have the following corollary.
\begin{corollary}
Let $(T(t))_{t\geq0}$ a $C_0-$semigroup, with generator $A$. Then:
\begin{enumerate}
  \item $e^{t\sigma_{ac}(A)}\subseteq \sigma_{ac}(T(t))$;
  \item $e^{t\sigma_{alc}(A)}\subseteq \sigma_{alc}(T(t))$.
\end{enumerate}
\end{corollary}
\begin{lemma}\label{l1}
Let $(T(t))_{t\geq0}$ a $C_0-$semigroup, $A$ its generator and  for all $x\in X$, $B_{\lambda}(t)x=\int_{0}^{t}e^{\lambda(t-s)}T(s)xds$,
$F_{\lambda}(t)x=\int_{0}^{t}e^{-\lambda s}B_{\lambda}(t)xds$.
Let $u\in D(A),v\in X$ such that $(\lambda-A)u=B_{\lambda}(t)v$. Then there exists $w\in D(A)$ such that:
  \begin{enumerate}
    \item $(\lambda -A)w=v$;
    \item $B_{\lambda}(t)w=u$;
    \item $\|w\|\leq (\|F_{\lambda}(t)\|+\|G_{\lambda}(t)\|)\max(\|u\|,\|v\|).$
  \end{enumerate}

\end{lemma}
\begin{proof}
Let $w=F_{\lambda}(t)v+G_{\lambda}(t)u \in D(A)$,
then $\|w\|\leq (\|F_{\lambda}(t)\|+\|G_{\lambda}\|)max \{\|u\|,\|v\|\}$. And we have:

\begin{eqnarray*}
  (\lambda -A)w &=& (\lambda-A)F_{\lambda}(t)v+(\lambda-A)G_{\lambda}(t)u \\
   &=& (I-B_{\lambda}(t)G_{\lambda}(t))v+(\lambda-A)G_{\lambda}(t)u  \mbox{  by (Lemma \ref{1} (2) )} \\
   &=&  v- G_{\lambda}(t)B_{\lambda}(t)v+G_{\lambda}(t)(\lambda-A)u \\
   &=& v
\end{eqnarray*}
\begin{eqnarray*}
  B_{\lambda}(t)w &=& B_{\lambda}(t)G_{\lambda}(t)u+B_{\lambda}(t)F_{\lambda}(t)v \\
   &=& (I-(\lambda-A)F_{\lambda}(t))u+B_{\lambda}(t)F_{\lambda}(t)v  \mbox{  by (Lemma \ref{1} (2) )}\\
   &=&  u-F_{\lambda}(t)(\lambda-A)u + F_{\lambda}(t)B_{\lambda}(t)v \\
   &=& u
\end{eqnarray*}
\end{proof}
\begin{lemma}\label{ll1}
Let $(T(t))_{t\geq0}$ a $C_0-$semigroup, with generator $A$. Then:\begin{enumerate}
                                                                    \item $K(B_{\lambda}(t))\cap K(\lambda -A)\subseteq K(e^{\lambda t}-T(t))$
                                                                    \item $C(B_{\lambda}(t))\cap C(\lambda -A)\subseteq C(e^{\lambda t}-T(t))$
                                                                    \item $S(B_{\lambda}(t))\cap S(\lambda -A)\subseteq S(e^{\lambda t}-T(t))$
                                                                  \end{enumerate}

\end{lemma}
\begin{proof}
 $1)$: We have $K(B_{\lambda}(t))\cap K(\lambda -A)\subseteq K(e^{\lambda t}-T(t))$.\\
Indeed : Let $x\in K(B_{\lambda}(t))\cap K(\lambda -A).$  Then there exists $(x_{i,0})_{i\geq0}\in D(A), (x_{0,j})_{j\geq0}\in X$ and $\delta>0$
such that $(\lambda -A)x_{i,0}=x_{i-1,0}, \|x_{i,0}\|\leq\delta^{i}\|x\|$ and $B_{\lambda}(t)x_{0,j}=x_{0,j-1},\|x_{0,j}\|\leq\delta^{j}\|x\|$.\\
We have $(\lambda-A)x_{1,0}=x_{0,0}=B_{\lambda}(t)x_{0,1}$, according to lemma \ref{l1} there exists a $x_{1,1}\in X$ such that :
\begin{center}
$(\lambda-A)x_{1,1}= x_{0,1}$ and  $B_{\lambda}(t)x_{1,1}=x_{1,0}$
\end{center}
and we have $(\lambda-A)x_{2,0}=x_{1,0}= B_{\lambda}(t)x_{1,1}$, lemma \ref{l1} implies there exits a $x_{2,1}$ such that  $(\lambda-A)x_{2,1}=x_{1,1}$, hence  $B_{\lambda}(t)x_{0,2}=x_{0,1}=(\lambda-A)x_{1,1}$ consequently  there exists $x_{1,2}$ such that $B_{\lambda}(t)x_{1,2}=x_{1,1}$, therefore
$(\lambda-A)x_{2,1}= B_{\lambda}(t)x_{1,2}=x_{1,1}$,  lemma \ref{l1} implies there exists $x_{2,2}$ such that
\begin{center}
$B_{\lambda}(t)x_{2,2}=x_{2,1}$ and $(\lambda-A)x_{2,2}=x_{1,2}$
\end{center}
Let $x_{2,2}=F_{\lambda}(t)x_{1,2}+ G_{\lambda}(t)x_{2,1}$, by induction we can construct a sequence $(x_{i,j})_{i,j\geq 0}$ defined by
$x_{i,j}=F_{\lambda}(t)x_{i-1,j}+ G_{\lambda}(t)x_{i,j-1}$, and we have  :

\begin{center}
$(\lambda-A)x_{i,j}=x_{i-1, j}$ and  $B_{\lambda}(t)x_{i,j}=x_{i,j-1}$
\end{center}
 and  $||x_{i,j}||\leq \delta ~~\max(||F_{\lambda}(t)||+ ||G_{\lambda}(t)||)^{i+j} ||x||$, for all $i,j\geq 1$

 Let $y_{i}=x_{i,i}$,  then $y_0=x_{0,0}=x$ and
 \begin{eqnarray*}
   (e^{\lambda t}-T(t))y_i &=& (\lambda-A)B_{\lambda}(t)y_i \\
   &=& (\lambda-A)B_{\lambda}(t)x_{i,i}\\
    &=& x_{i-1, j-1}=y_{i-1}
 \end{eqnarray*}
  and $||y_{i}||=||x_{i,i}||\leq \beta ^{i} ||x||$ where $\beta > 0$, hence $x\in  K(e^{\lambda t}-T(t))$.\\
 $2)$: Similar to $1)$\\
 $3)$: Let $x\in S(B_{\lambda}(t))\cap S(\lambda -A).$  Then there exists $(x_{i,0})_{i\geq 0}\in D(A), (x_{0,j})_{j\geq0}\in X$,
 $x_{0,0}=0$, such that $(\lambda -A)x_{i,0}=x_{i-1,0},~~ \sup_{i}\|x_{i,0}\|^{\frac{1}{i}}<\infty$ and $B_{\lambda}(t)x_{0,j}=x_{0,j-1},~~\sup_j \|x_{0,j}\|^{\frac{1}{j} }<\infty$, by the same arguments in $1)$,  we can show that
  $x\in S(e^{\lambda t}-T(t))$.

\end{proof}

\begin{theorem}
Let $(T(t))_{t\geq0}$ a $C_0-$semigroup, with generator $A$. Then

\begin{enumerate}
  \item $K(e^{\lambda t}-T(t))\cap K(B_{\lambda}(t))= K(\lambda-A)\cap K(B_{\lambda}(t)),$
  \item $C(e^{\lambda t}-T(t))\cap C(B_{\lambda}(t))= C(\lambda-A)\cap C(B_{\lambda}(t)).$

\end{enumerate}
\end{theorem}
\begin{proof}
Corollary \ref{c1} and lemma \ref{ll1} gives the result.
\end{proof}

\section{Stability Results.}
Let $(T(t))_{t\geq 0}$ a $C_0$-semigroup on $X$ with infinitesimal generator $A$.
$\{T(t)\}_{t\geq 0}$ is said to be strongly stable if $ \displaystyle{\lim_{t\rightarrow \infty}}||T(t)x||=0$ for all $x\in X$.
We say that $(T (t))_{t\geq0}$ is uniformly stable if $ \displaystyle{\lim_{t\rightarrow \infty}}||T(t)||=0$.\\
In \cite{Mulk}, A. Elkoutri and M. A. Taoudi showed that $(T (t))t\geq0$ is strongly stable if  $\sigma_{K}(A)\cap i\mathbb{R}=\emptyset$.

 In the following, we give a stability result for strongly continuous semigroups using the local  spectrum:
\begin{theorem}
Let $A$ be the generator of a bounded strongly continuous semigroup $(T(t))_{t\geq0}$.
 If for all $x\in X$, $\sigma_{A}(x)\cap i\mathbb{R}=\emptyset$,
 then $(T(t))_{t\geq 0}$ is strongly stable.
\end{theorem}
\begin{proof}
  If $\sigma_{A}(x)\cap i\mathbb{R}=\emptyset$ for all $x\in X$,  then
 $$\emptyset=\displaystyle{\bigcup_{x\in X}}(\sigma_A(x)\cap i\mathbb{R})=\displaystyle{\bigcup_{x\in X}}\sigma_A(x)\cap i\mathbb{R}=\sigma_{su}(A)\cap i\mathbb{R}.$$
As $\sigma_{K}(A)\cap i\mathbb{R}\subseteq \sigma_{su}(A)\cap i\mathbb{R}=\emptyset$, then  $\sigma_{K}(A)\cap i\mathbb{R}=\emptyset.$
According to \cite[corollary 2.1]{Mulk},   $(T(t))_{t\geq 0}$ is strongly stable
\end{proof}
\begin{theorem}
Let $A$ be the generator of a bounded strongly continuous semigroup $(T(t))_{t\geq0}$.
Then, the following assertions are equivalent:\begin{enumerate}

                                                \item $(T(t))_{t\geq0}$ is uniformly stable;
                                                \item  for all $x\in X$, there exist $t_0>0$ such that $\sigma_{T(t_0)}(x)\cap \Gamma=\emptyset$
                                              \end{enumerate}
 where $\Gamma$ stands for the unit circle of $\mathbb{C}$.
\end{theorem}
\begin{proof}
According to \cite[corollary 2.2]{Mulk} and \cite[Theorem 3.2]{Mulko} ,
it suffices to show that  $\sigma_{(T(t_0))}(x)\cap \Gamma=\emptyset$ implies that  $\sigma_{K}(T(t_0))\cap \Gamma=\emptyset$.
Indeed:  If $\sigma_{(T(t_0))}(x)\cap \Gamma=\emptyset$ for all $x\in X$,  then $$\emptyset=\displaystyle{\bigcup_{x\in X}}(\sigma_{T(t_0)}(x)\cap \Gamma)=\displaystyle{\bigcup_{x\in X}}\sigma_{(T(t_0))}(x)\cap \Gamma=\sigma_{su}(T(t_0))\cap \Gamma.$$
As $\sigma_{K}(T(t_0))\cap \Gamma\subseteq \sigma_{su}(T(t_0))\cap \Gamma=\emptyset$, then  $\sigma_{K}(T(t_0))\cap \Gamma=\emptyset.$

\end{proof}

\end{document}